\documentclass[english,10pt, a4paper, final]{article}
\usepackage[T1]{fontenc}
\usepackage[latin9]{inputenc}
\usepackage{geometry}
\PassOptionsToPackage{obeyFinal}{todonotes}
\usepackage{color}
\usepackage{babel}
\usepackage{units}
\usepackage{mathtools}
\usepackage{todonotes}
\usepackage{amsmath}
\usepackage{amsthm}
\usepackage{amssymb}
\usepackage{graphicx}
\usepackage[numbers]{natbib}
\usepackage{microtype}
\usepackage[unicode=true,
 bookmarks=true,bookmarksnumbered=false,bookmarksopen=false,
 breaklinks=false,pdfborder={0 0 0},pdfborderstyle={},backref=page,colorlinks=true]
 {hyperref}
\hypersetup{pdftitle={Danamic programming},
 pdfauthor={Alois Pichler},
 citecolor=blue, urlcolor= darkgray, linkcolor= darkgray}

\makeatletter
\numberwithin{equation}{section}
\theoremstyle{plain}
\newtheorem{thm}{\protect\theoremname}[section]
\theoremstyle{definition}
\newtheorem{defn}[thm]{\protect\definitionname}
\theoremstyle{plain}
\newtheorem*{prop*}{\protect\propositionname}
\theoremstyle{plain}
\newtheorem{prop}[thm]{\protect\propositionname}
\theoremstyle{remark}
\newtheorem{rem}[thm]{\protect\remarkname}
\theoremstyle{plain}
\newtheorem{cor}[thm]{\protect\corollaryname}
\theoremstyle{definition}
\newtheorem{example}[thm]{\protect\examplename}
\theoremstyle{plain}
\newtheorem{lem}[thm]{\protect\lemmaname}

\usepackage[notcite,notref]{showkeys}	
\usepackage{dsfont}   
\usepackage{newtxtext, newtxmath}	
\usepackage{microtype}
\usepackage{siunitx}
\usepackage{xcolor}		

\mathtoolsset{showonlyrefs}
\usepackage{doi}

\DeclareMathOperator{\E}{{\mathds E}}
\DeclareMathOperator{\var}{var}		
\DeclareMathOperator{\cov}{cov}		
\DeclareMathOperator{\corr}{corr}		

\newcommand{\one}{{\mathds 1}} 		

\@ifundefined{showcaptionsetup}{}{%
 \PassOptionsToPackage{caption=false}{subfig}}
\usepackage{subfig}
\makeatother

\providecommand{\corollaryname}{Corollary}
\providecommand{\definitionname}{Definition}
\providecommand{\examplename}{Example}
\providecommand{\lemmaname}{Lemma}
\providecommand{\propositionname}{Proposition}
\providecommand{\remarkname}{Remark}
\providecommand{\theoremname}{Theorem}

\begin{document}
\title{Uniform Function Estimators in Reproducing~Kernel~Hilbert~Spaces}
\author{Paul Dommel\footnotemark[1] \and Alois Pichler\thanks{Technische Universit\"at Chemnitz, Faculty of mathematics. 90126 Chemnitz,
Germany\protect \\
DFG, German Research Foundation \textendash{} Project-ID 416228727
\textendash{} SFB~1410\protect \\
\protect\includegraphics[width=0.8em]{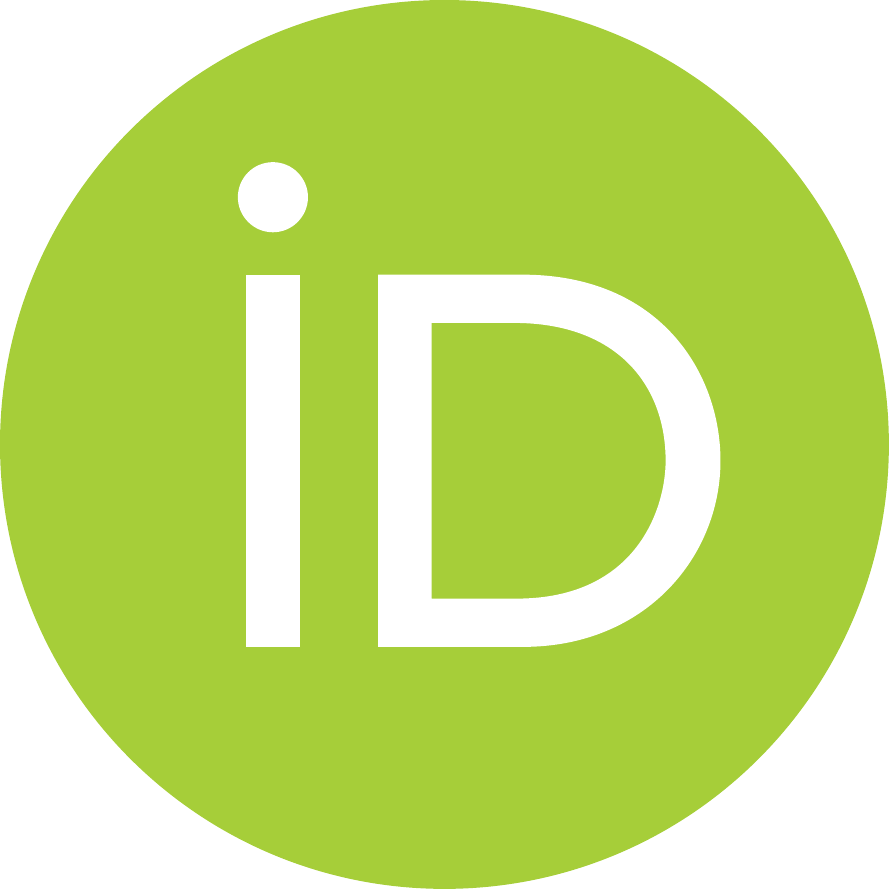}
\protect\href{https://orcid.org/0000-0001-8876-2429}{orcid.org/0000-0001-8876-2429}.
Contact: \protect\href{mailto:alois.pichler@math.tu-chemnitz.de}{alois.pichler@math.tu-chemnitz.de}} }
\maketitle
\begin{abstract}
This paper addresses the problem of regression to reconstruct functions,
which are observed with superimposed errors at random locations. We
address the problem in reproducing kernel Hilbert spaces. It is demonstrated
that the estimator, which is often derived by employing Gaussian random
fields, converges in the mean norm of the reproducing kernel Hilbert
space to the conditional expectation and this implies local and uniform
convergence of this function estimator. By preselecting the kernel,
the problem does not suffer from the curse of dimensionality.

The paper analyzes the statistical properties of the estimator. We
derive convergence properties and provide a conservative rate of convergence
for increasing sample sizes. \medskip{}

\noindent \textbf{Keywords:} Reproducing kernel Hilbert spaces {\tiny\textbullet} positive
definite functions {\tiny\textbullet}  Gramian matrix~{\tiny\textbullet}  Mercer kernels {\tiny\textbullet}  Statistical
learning theory

\noindent \textbf{Classification:} 62G05, 62G08, 62G20 
\end{abstract}

\section{Introduction}

This paper addresses the problem of regression and approximation,
nowadays occasionally often associated with the term statistical learning.
The specific estimator we consider is based on kernel functions. We
investigate the estimator's convergence properties in the the genuine
and most natural norm, the norm induced by the kernel function itself. 

The estimator is often derived by involving Gaussian random fields
and is central in support vector machines as well, an additional motivational
point to investigate its specific properties. Here, the estimator
is often inferred with least squares errors and by involving a regularization
term based on a reproducing kernel Hilbert space. The literature frequently
employs loss and risk functionals, and involves an $L^{2}$\nobreakdash-error
to investigate this estimator. Our results complement these research
directions by adding the natural, genuine norm. They enable us to
establish uniform convergence of the estimator by moderately regularizing
the objectives. This uniform convergence is indeed essential and crucial
for applications in stochastic optimization. 

Explicit convergence rates are presented for increasing sample sizes.
The results and convergence rates correspond to other rates known
from non-parametric statistics, particularly to density estimation
when employing the mean (integrated) squared error. Starting with
a fixed kernel, the results presented here do not depend on the dimension
of the design space, so they do not suffer from what is occasionally
addressed by the catchphrase \emph{curse of dimensionality}.

\medskip{}

\citet{Cucker2007} provide an introduction to approximation theory
in a random framework. The excellent book \citet[Section~2.3]{Bishop2006}
gives very concrete applications in statistical learning theory, while
\citet{Wendland2004} provide the mathematical foundations for approximations
in reproducing kernel Hilbert spaces. The monograph \citet{SteinwartChristmann}
introduces to support vector machines, which employ kernel functions
similarly to our approach presented below. A study, comparably to
ours but employing a simpler norm, is \citet{ZhangDuchiWainwright}.
\citet{Caponnetto2006} provide the state of the art for an analysis
in $L^{2}$ involving the kernel operator, see also \citet{Gyoerfi2002}. 

\paragraph{Outline of the paper. }

The following Section~\ref{sec:Regularization} repeats elements
from reproducing kernel Hilbert spaces, which are of importance throughout
this paper. Section~\ref{sec:Approx} introduces the elementary estimator,
which is employed in statistical learning. Sample average approximation
(Section~\ref{sec:SAA}) address this estimator with random samples
from both dimensions and Section~\ref{eq:56} reveals related statistical
results. The Sections~\ref{sec:InNorm} and~\ref{sec:Convergence}
derive our main results, which is, for short, convergence of the sample
average optimizer in mean norm and weak consistency (Section~\ref{sec:Consistency})
of this estimator. Section~\ref{sec:Summary} concludes with a summary.

\section{\label{sec:Regularization}Regularization with reference to reproducing
kernel Hilbert spaces}

Throughout we shall expose the problem on the design space $\mathcal{X}$,
an arbitrary set for which we require more structure later; most typically,~$\mathcal{X}$
is a subset of~$\mathbb{R}^{d}$. Let $(X_{i},f_{i})$, $i=1,\dots,n$,
be independent, identically distributed observations in $\mathcal{X}\times\mathbb{R}$
with joint probability measure~$\mathcal{\rho}$. For a kernel function
$k\colon\mathcal{X}\times\mathcal{X}\to\mathbb{R}$ we consider the
estimator 
\begin{equation}
\hat{f}_{n}(\cdot)=\frac{1}{n}\sum_{i=1}^{n}k(\cdot,X_{i})\,\hat{w}_{i},\label{eq:Estimator}
\end{equation}
where the weights $\hat{w}_{i}$ satisfy the system of linear equations
\begin{equation}
\lambda_{n}\,\hat{w}_{i}+\frac{1}{n}\sum_{j=1}^{n}k(X_{i},X_{j})\,\hat{w}_{j}=f_{i},\qquad i=1,\dots,n,\label{eq:17}
\end{equation}
for some parameter~$\lambda_{n}$.\footnote{Note that $\hat{f}_{n}$ interpolates the data, $\hat{f}_{n}(X_{i})=f_{i}$,
$i=1,\dots,n$, for the particular choice $\lambda_{n}=0$.} In what follows we derive this estimator first by employing Gaussian
random fields and kernel ridge regression from support vector machines
and then investigate and expose its convergence properties. Specifically,
we identify and characterize the function~$f$ so that 
\begin{equation}
\E\|\hat{f}_{n}(\cdot)-f(\cdot)\|^{2}\to0\label{eq:Master}
\end{equation}
as $n\to\infty$, where $\|\cdot\|$ is a norm and $\lambda_{n}$
is chosen adequately; above all, we derive results for the norm of
the reproducing kernel Hilbert space associated with the kernel function.
We will also infer convergence results for $L^{2}$ and\textemdash most
importantly\textemdash for uniform function approximations, as point
evaluations are continuous in the kernel norm.

\subsection{Gaussian random fields}

As an initial motivation for the estimator~\eqref{eq:Estimator}
consider a zero mean Gaussian random field~$f$ on $\mathcal{X}$
with covariance function $k\colon\mathcal{X}\times\mathcal{X}\to\mathbb{R}$,
that is, $k(x,y)=\cov\big(f(x),f(y)\big)$. For a signal plus noise
model with observations 
\[
f_{i}=f(x_{i})+\varepsilon_{i},
\]
the joint distribution, including $x$ to the observation points $X=(x_{1},\dots,x_{n})$,
is 
\[
\begin{pmatrix}f(x)\\
f
\end{pmatrix}\sim\mathcal{N}\left(\begin{pmatrix}0\\
0
\end{pmatrix},\ \begin{pmatrix}k(x,x) & k(x,X)\\
k(X,x) & k(X,X)+\lambda
\end{pmatrix}\right),
\]
where $\varepsilon\sim\mathcal{N}(0,\lambda)$ is the independent
error and where we use the compact vector notation $f\coloneqq(f_{1},\dots,f_{n})^{\top}$
and $k(x,X)\coloneqq\big(k(x,x_{1}),\dots,k(x,x_{n})\big)$ for the
entry of the covariance matrix; the other entries are defined analogously.
With this, the conditional distribution is 
\[
f(x)\mid\big(f(X)=f\big)\sim\mathcal{N}\big(\hat{\mu}(x),\,\hat{K}(x)\big),
\]
where 
\begin{equation}
\hat{\mu}(x)\coloneqq k(x,X)\big(k(X,X)+\lambda\big)^{-1}f(X)\label{eq:1}
\end{equation}
is the mean and the variance is
\begin{equation}
\hat{K}(x)\coloneqq k(x,x)-k(x,X)\big(k(X,X)+\lambda\big)^{-1}k(X,x),\label{eq:Bandwidth}
\end{equation}
see \citet[Theorem~13.1]{Shiryaev1996} or \citet[Section~2.3]{Bishop2006}.
Expanding~\eqref{eq:1} and setting $\hat{f}_{n}(x)\coloneqq\hat{\mu}(x)$
reveals the initial estimator~\eqref{eq:Estimator} for variance~$\lambda$
rescaled. Figure~\ref{fig:Gaussian} displays an example of the estimator~$\hat{f}_{n}(\cdot)$
together with the range $\pm\sqrt{\hat{K}(\cdot)}$ from~\eqref{eq:Bandwidth}.

\subsection{Reproducing kernel Hilbert space }

Every estimator $\hat{f}_{n}(\cdot)$ in~\eqref{eq:Estimator} is
an element in the reproducing kernel Hilbert space spanned by the
functions $k(\cdot,y)$, $y\in\mathcal{X}$. While introducing the
notation for reproducing kernel Hilbert spaces here we briefly recall
major properties, which are essential in our following exposition.
For a general discussion on reproducing kernel Hilbert spaces we may
refer to \citet[Chapter~1]{Mandrekar2016}. 
\begin{defn}
The kernel is a symmetric and positive definite function $k\colon\mathcal{X}\times\mathcal{X}\to\mathbb{R}$.
On the linear span $\{k(\cdot,x)\colon\mathcal{X}\to\mathbb{R}\mid x\in\mathcal{X}\}$
of functions on $\mathcal{X}$, the inner product is defined by
\begin{equation}
\bigl\langle k(\cdot,x)\mid k(\cdot,y)\bigr\rangle_{k}\coloneqq k(x,y).\label{eq:InnerProduct}
\end{equation}
The reproducing kernel Hilbert space, denoted $\big(\mathcal{H}_{k},\|\cdot\|_{k}\big)$,
is the completion with respect to the norm $\|f\|_{k}^{2}\coloneqq\left\langle f\mid f\right\rangle _{k}$
induced by the inner product~\eqref{eq:InnerProduct}.
\end{defn}

Most importantly, point evaluations are continuous linear functions
in reproducing kernel Hilbert spaces. Indeed, finite linear combinations
$f(\cdot)=\sum_{i=1}^{n}k(\cdot,x_{i})\,w_{i}$ are dense in $\mathcal{H}_{k}$
and it follows from~\eqref{eq:InnerProduct} that 
\begin{equation}
\bigl\langle k(\cdot,x)\big|\,f\bigr\rangle_{k}=\sum_{i=1}^{n}w_{i}\,\bigl\langle k(\cdot,x)\big|\,k(\cdot,x_{i})\bigr\rangle_{k}=\sum_{i=1}^{n}w_{i}\,k(x,x_{i})=f(x).\label{eq:Stetig}
\end{equation}

Although more general settings are easily possible, in what follows
we convene to address only continuous and uniformly bounded kernel
functions~$k$. We associate the following Hilbert\textendash Schmidt
integral operator~$K$ with a kernel~$k$. 
\begin{defn}[Design measure]
\label{def:DesignMeasure}Let $\mathcal{X}$ be a measure space.
The marginal measure $P(\cdot)\coloneqq\rho(\cdot\times\mathbb{R})$
is the \emph{design measure}. 
\end{defn}

\begin{defn}
Let $k$ be a kernel. The operator $K\colon L^{2}(\mathcal{X})\to L^{2}(\mathcal{X})$
is 
\begin{equation}
K\,w(x)\coloneqq\int_{\mathcal{X}}k(x,y)\,w(y)\,P(dy),\label{eq:K}
\end{equation}
where $w\in L^{2}(\mathcal{X})$.
\end{defn}

\begin{prop*}
The operator~$K$ is self-adjoint and positive definite with respect
to the standard inner product 
\[
\left\langle f\mid g\right\rangle \coloneqq\int_{\mathcal{X}}f(z)\cdot g(z)\,P(dz)
\]
on $\big(L^{2},\|\cdot\|_{2}\big)$. The operator is positive definite
and bounded with norm 
\[
\|K\colon L^{2}\to L^{2}\|^{2}\le\iint_{\mathcal{X}^{2}}k(x,y)^{2}\,P(dx)P(dy).
\]
\end{prop*}
\begin{proof}
The assertion is a consequence of the Cauchy\textendash Schwarz inequality.
\end{proof}
\begin{prop}
\label{prop:23}It holds that $\|k(\cdot,x)\|_{k}^{2}=k(x,x)$, 
\begin{equation}
\left\langle Kw\mid f\right\rangle _{k}=\left\langle w\mid f\right\rangle \quad\text{and}\quad\left\Vert Kw\right\Vert _{k}^{2}=\left\langle w\mid Kw\right\rangle .\label{eq:kNorm}
\end{equation}
\end{prop}

\begin{proof}
The functions $f(\cdot)=\sum_{i=1}^{n}w_{i}^{\prime}\,k(\cdot,x_{i})$
are dense in $\mathcal{H}_{k}$. By linearity, 
\begin{align*}
\left\langle Kw\mid f\right\rangle _{k} & =\sum_{i=1}^{n}\,w_{i}^{\prime}\int_{\mathcal{X}}\big<k(\cdot,y)\mid k(\cdot,x_{i})\big>_{k}w(y)\,P(dy)\\
 & =\int_{\mathcal{X}}\sum_{i=1}^{n}\,w_{i}^{\prime}\,k(y,x_{i})\,w(y)\,P(dy)\\
 & =\int_{\mathcal{X}}f(y)\,w(y)\,P(dy)\\
 & =\left\langle w\mid f\right\rangle .
\end{align*}
The other assertions are immediate. 
\end{proof}
\begin{rem}[Mercer\footnote{The initial publication is notably due to Schmidt, see \citet{ErhardSchmidt},
and not Mercer.} and the kernel trick]
\label{rem:Mercer}The operator $K$ is compact with $k(x,y)=\sum_{\ell=1}^{\infty}\sigma_{\ell}^{2}\,\phi_{\ell}(x)\,\phi_{\ell}(y)$,
where $\sigma_{\ell}^{2}$ is the eigenvalue corresponding to the
eigenfunction~$\phi_{\ell}(\cdot)$. In this setting, the operator
$K^{\nicefrac{1}{2}}$ is $K^{\nicefrac{1}{2}}f=\sum_{\ell=1}^{\infty}\sigma_{\ell}\,\phi_{\ell}\left\langle \phi_{\ell}\mid f\right\rangle $
(with $\sigma_{\ell}\ge0$), see \citet[Theorem~VI.23]{Reed1980}. 
\end{rem}

\begin{prop}[$K^{\nicefrac{1}{2}}\colon L^{2}\to\mathcal{H}_{k}$ is an isometry]
\label{prop:Norm}It holds that $\|K^{\nicefrac{1}{2}}f\|_{k}=\|f\|_{2}$
and $\|f\|_{2}\le\|K\|^{\nicefrac{1}{2}}\cdot\|f\|_{k}$.
\end{prop}

\begin{proof}
The assertion is a consequence of Mercer's theorem, cf.\ \citet{KoenigOperators}
and \citet[Corollary~4]{HeinBousquet}. However, for $f=K^{\nicefrac{1}{2}}w$
it follows from the preceding proposition that 
\[
\|K^{\nicefrac{1}{2}}f\|_{k}^{2}=\|Kw\|_{k}^{2}=\left\langle w\mid Kw\right\rangle =\left\langle K^{\nicefrac{1}{2}}w\mid K^{\nicefrac{1}{2}}w\right\rangle =\|f\|_{2}^{2}.
\]
With~\eqref{eq:kNorm} we have further that 
\[
\|f\|_{2}^{2}=\big<K^{\nicefrac{1}{2}}w\mid K^{\nicefrac{1}{2}}w\big>=\bigl\langle w\mid Kw\bigr\rangle\le\|K\|\,\|w\|_{2}^{2}=\|K\|\,\|K^{\nicefrac{1}{2}}w\|_{k}^{2}=\|K\|\,\|f\|_{k}^{2},
\]
as $K$ is self-adjoint. Hence the assertion.
\end{proof}
\begin{thm}[Continuity of the operator $K$]
It holds that $\|K\colon\mathcal{H}_{k}\to\mathcal{H}_{k}\|\le\|K\colon L_{2}\to L_{2}\|$.
\end{thm}

\begin{proof}
With~\eqref{eq:kNorm} and Proposition~\ref{prop:Norm}, $\|Kf\|_{k}^{2}=\left\langle f\mid Kf\right\rangle \le\|K\|\,\|f\|_{2}^{2}\le\|K\|^{2}\|f\|_{k}^{2}$
and hence the assertion.
\end{proof}
We have seen in~\eqref{eq:Stetig} that point evaluations are linear
functionals. We shall conclude here by relating these norms to uniform
convergence.
\begin{prop}
The point evaluation is continuous; indeed, $|f(x)|\le\sqrt{k(x,x)}\,\|f\|_{k}$
for all $x\in\mathcal{X}$ and $f\in\mathcal{H}_{k}$. Further,\footnote{\label{fn:Support}The support of the measure $P$ is $\operatorname{supp}P\coloneqq\bigcap\big\{ A\colon A\text{ is closed and }P(A)=1\big\}\subset\mathcal{X}$,
cf.\ \citet{Rueschendorf}.}
\begin{equation}
\|f\|_{\infty}\le\|f\|_{k}\cdot\sup_{x\in\operatorname{supp}P}\sqrt{k(x,x)},\label{eq:Point}
\end{equation}
where $\|f\|_{\infty}\coloneqq\sup_{x\in\operatorname{supp}P}|f(x)|$.
\end{prop}

\begin{proof}
The statement is immediate from~\eqref{eq:Stetig}, as 
\[
|f(x)|=\big|\big<k(\cdot,x)\big|\,f\big>_{k}\big|\le\|k(\cdot,x)\|_{k}\,\|f\|_{k}=\sqrt{k(x,x)}\,\|f\|_{k}
\]
by the Cauchy\textendash Schwartz inequality and Proposition~\ref{prop:23}. 
\end{proof}

\section{\label{sec:Approx}The genuine approximation problem}

In what follows we characterize the estimator~\eqref{eq:Estimator}
by involving a stochastic optimization problem. We consider the problem
first in its continuous form and relate it to the data subsequently.

By the disintegration theorem (see \citet{Dellacherie1978} or \citet{Ambrosi2005})
there is a family of measures $\rho(\cdot\mid x)\colon\mathcal{B}(\mathcal{X})\to[0,1]$,
$x\in\mathcal{X}$, on the Borel sets $\mathcal{B}(\mathcal{X})$
so that 
\[
\rho(A\times B)=\int_{A}\rho(B|\,x)\,P(dx),\quad A\subset\mathcal{X},B\subset\mathbb{R}\text{ measurable},
\]
where $P(\cdot)=\rho(\cdot\times\mathbb{R})$ is the design measure,
see Definition~\ref{def:DesignMeasure}. For a random variable $(X,f)$
with law~$\rho$ we recall the notational variants 
\[
\E g(X,f)=\iint_{\mathcal{X}\times\mathbb{R}}g(x,f)\,\rho(dx,df)=\int_{\mathcal{X}}g(x,f)\,\rho(df|\,x)\,P(dx)=\E\E\big(g(X,f)|\,X\big),
\]
where $g$ is measurable and 
\[
\E\big(g(x,f)|\,x\big)=\int_{\mathcal{X}}g(x,f)\,\rho(df|\,x),\qquad x\in\mathcal{X},
\]
is the conditional expectation.

\subsection{The continuous problem}

For the random variable $(X,f)$ with values in $\mathcal{X}\times\mathbb{R}$,
law $\rho$ and $f\in L^{2}$ consider the optimization problem 
\begin{equation}
\min_{f_{\lambda}(\cdot)\in\mathcal{H}_{k}}\E\big(f-f_{\lambda}(X)\big)^{2}+\lambda\left\Vert f_{\lambda}\right\Vert _{k}^{2},\label{eq:Genuine}
\end{equation}
where $\lambda>0$ is a fixed regression parameter. The objective~\eqref{eq:Genuine}
is strictly convex in $\|\cdot\|_{k}$, so that convergence can be
established for both, the optimal value and its optimizer, provided
that $\lambda>0$ is fixed.

The random variable $f_{\lambda}(X)$ is measurable with respect to
$\sigma(X)$, the $\sigma$\nobreakdash-algebra generated by~$X$,
and the random variable $\E(f\mid X)$ is the projection of~$f$
onto the closed subspace $L^{2}\big(\sigma(X)\big)$, see \citet{Kallenberg2002Foundations}.
By the Pythagorean theorem, the objective in the preceding problem
thus is equivalently 
\[
\min_{f_{\lambda}(\cdot)}\E\big(f-\E(f\mid X)\big)^{2}+\E\big(\E(f\mid X)-f_{\lambda}(X)\big)^{2}+\lambda\left\Vert f_{\lambda}\right\Vert _{k}^{2}.
\]
It follows from the Doob\textendash Dynkin lemma that there is a Borel
function $f_{0}\colon\mathcal{X}\to\mathbb{R}$ so that $\E(f\mid X)=f_{0}(X)$.
We follow the convention and denote this function also as 
\begin{equation}
f_{0}(x)=\E(f\mid X=x).\label{eq:Ef}
\end{equation}
The orthogonality relation characterizing~$f_{0}$ is 
\begin{equation}
\E\big(f-f_{0}(X)\big)g(X)=0,\label{eq:Orthogonal}
\end{equation}
where $g\colon\mathcal{X}\to\mathbb{R}$ is any measurable test function.
The objective of the optimization problem~\eqref{eq:Genuine} thus
is 
\begin{equation}
\vartheta^{*}\coloneqq\E\big(f-f_{0}(X)\big)^{2}+\min_{f_{\lambda}(\cdot)}\E\big(f_{0}(X)-f_{\lambda}(X)\big)^{2}+\lambda\left\Vert f_{\lambda}\right\Vert _{k}^{2},\label{eq:Theta*}
\end{equation}
where the quantity $\E\big(f-f_{0}(X)\big)^{2}$ is the \emph{irreducible
error}.
\begin{rem}
\label{rem:Closed}We  note that $f_{0}\in L^{2}\big(\sigma(X)\big)$,
but $f_{0}$~is \emph{not necessarily} in $\mathcal{H}_{k}$.
\end{rem}

\begin{thm}
\label{thm:Optimal}The solution of the optimization problem~\eqref{eq:Genuine}
is 
\begin{equation}
f_{\lambda}=Kw_{\lambda},\label{eq:64}
\end{equation}
where $(\lambda+K)w_{\lambda}=f_{0}$; the objective is
\begin{align}
\vartheta^{*} & =\left\Vert f-f_{0}\right\Vert ^{2}+\left\Vert f_{0}-f_{\lambda}\right\Vert ^{2}+\lambda\left\Vert f_{\lambda}\right\Vert _{k}^{2}\nonumber \\
 & =\left\Vert f-f_{0}\right\Vert ^{2}+\lambda\left\langle w_{\lambda}\mid Kw_{\lambda}\right\rangle +\lambda^{2}\left\Vert w_{\lambda}\right\Vert ^{2}.\label{eq:37}
\end{align}
\end{thm}

\begin{proof}
With~\eqref{eq:kNorm} we may rewrite the objective in~\eqref{eq:Theta*}
by $g(w_{\lambda})\coloneqq\left\Vert f_{0}-Kw_{\lambda}\right\Vert ^{2}+\lambda\left\langle w_{\lambda}\mid Kw_{\lambda}\right\rangle $.
Now note that 
\begin{align*}
g(w_{\lambda}+h)-g(w_{\lambda}) & =\left\langle f_{0}-Kw_{\lambda}-Kh)\mid f_{0}-Kw_{\lambda}-Kh\right\rangle +\lambda\left\langle w_{\lambda}+h\mid K(w_{\lambda}+h)\right\rangle \\
 & \qquad-\left\langle f_{0}-Kw_{\lambda}\mid f_{0}-Kw_{\lambda}\right\rangle -\lambda\left\langle w_{\lambda}\mid Kw_{\lambda}\right\rangle \\
 & =-\left\langle Kh\mid f_{0}-Kw_{\lambda}\right\rangle -\left\langle f_{0}-Kw_{\lambda}\mid Kh\right\rangle +\left\langle Kh\mid Kh\right\rangle \\
 & \qquad+\lambda\left\langle h\mid Kw_{\lambda}\right\rangle +\lambda\left\langle h\mid Kw_{\lambda}\right\rangle +\lambda\left\langle h\mid Kh\right\rangle \\
 & =-2\left\langle Kh\mid f_{0}-Kw_{\lambda}-\lambda w_{\lambda}\right\rangle +\left\langle Kh\mid Kh\right\rangle +\lambda\left\langle h\mid Kh\right\rangle 
\end{align*}
as $K$ is self-adjoint. The first, linear term vanishes if $(\lambda+K)w_{\lambda}=f_{0}$,
and the second is quadratic in~$h$ \textendash{} hence the infimum
and the first assertion. For the objective~\eqref{eq:37} note that
$f_{0}-f_{\lambda}=\lambda\,w_{\lambda}$, see also~\eqref{eq:36}
below.
\end{proof}
\begin{cor}[Characterization of the coefficient function]
\label{cor:33}Suppose that 
\begin{equation}
(\lambda+K)w_{\lambda}=f_{0},\label{eq:35}
\end{equation}
then 
\begin{equation}
f_{\lambda}\coloneqq Kw_{\lambda}=(\lambda+K)^{-1}Kf_{0}\label{eq:34}
\end{equation}
 solves the Fredholm equation of the second kind $(\lambda+K)f_{\lambda}=Kf_{0}$
and it holds that
\begin{equation}
f_{0}-f_{\lambda}=\lambda\,w_{\lambda}.\label{eq:36}
\end{equation}
\end{cor}

\begin{proof}
Apply $K$ to~\eqref{eq:35} to get $\lambda Kw_{\lambda}+KKw_{\lambda}=Kf_{0}$,
that is, $(\lambda+K)f_{\lambda}=Kf_{0}$.
\end{proof}
\begin{rem}
It follows from~\eqref{eq:34} that $f_{\lambda}\in\mathcal{H}_{k}$,
even more, $f_{\lambda}$ is in the image of $K$, although $f_{0}$
is \emph{not necessarily} in $\mathcal{H}_{k}$ (cf.\ Remark~\ref{rem:Closed}).
\end{rem}

The distance of the solution $f_{\lambda}$ to the function $f_{0}$
will be of importance in what follows. We have the following general
result.
\begin{prop}
\label{prop:f0}Suppose that $f_{0}$ is in the range of $K$. Then
there is a constant $C_{0}>0$ so that 
\[
\|f_{0}-f_{\lambda}\|_{k}\le C_{0}\,\lambda.
\]
\end{prop}

\begin{proof}
As $f_{0}$ is in the range of $K$ there is some $w_{0}\in\mathcal{H}_{k}$
so that $f_{0}=K\,w_{0}$. For $w_{0}\in\mathcal{H}_{k}$ there is
further $w\in L^{2}$ so that $w_{0}=K^{\nicefrac{1}{2}}w$ by Proposition~\ref{prop:Norm}.
With~\eqref{eq:35} it holds that 
\[
w_{\lambda}=(\lambda+K)^{-1}f_{0}=(\lambda+K)^{-1}Kw_{0}=K^{\nicefrac{1}{2}}(\lambda+K)^{-1}Kw
\]
and thus, with Proposition~\ref{prop:Norm} again, 
\[
\|w_{\lambda}\|_{k}=\|(\lambda+K)^{-1}Kw\|_{2}\le\|w\|_{2},
\]
as $(\lambda+K)^{-1}K\le\one$ in Loewner order. With~\eqref{eq:36}
it follows that $\|f_{0}-f_{\lambda}\|_{k}=\lambda\|w_{\lambda}\|_{k}\le\lambda\|w\|_{2}$
and thus the assertion with the constant $C_{0}\coloneqq\|w\|_{2}=\|w_{0}\|_{k}$.
\end{proof}
The following corollary to Corollary~\ref{cor:33} provides the weight
functions with respect to the usual Lebesgue measure. We provide this
statement as it particularly useful to solving the Fredholm integral
equation~\eqref{eq:35} numerically (by employing the Nystr�m method,
for example, cf.\ \citet{BachLowRank}) to make the function~$f_{\lambda}$
available for computational purposes.
\begin{cor}[Coefficient function for measures with a density]
Suppose that $P$ has a density $p(\cdot)$ with respect to the Lebesgue
measure, $P(dx)=p(x)dx$, and the coefficient function $\tilde{w}_{\lambda}(\cdot)$
satisfies 
\begin{equation}
\lambda\,\tilde{w}_{\lambda}(x)+p(x)\cdot\int_{\mathcal{X}}k(x,y)\,\tilde{w}_{\lambda}(y)\,dy=p(x)\cdot g_{0}(x).\label{eq:62}
\end{equation}
Then the function $g_{\lambda}(\cdot)\coloneqq\int_{\mathcal{X}}k(\cdot,x)\,\tilde{w}_{\lambda}(x)\,dx$
solves the integral equation
\[
(\lambda+K)g_{\lambda}=Kg_{0}.
\]
\end{cor}

\begin{proof}
Multiply equation~\eqref{eq:62} by $k(y,x)$ and integrate with
respect to $dx$ to get 
\[
\lambda\int_{\mathcal{X}}k(y,x)\,\tilde{w}_{\lambda}(x)\,dx+\int_{\mathcal{X}}k(y,x)\cdot\int_{\mathcal{X}}k(x,z)\,\tilde{w}_{\lambda}(z)\,dz\,p(x)dx=\int_{\mathcal{X}}k(y,x)\,g_{0}(x)\,p(x)dx.
\]
This is 
\[
\lambda\,g_{\lambda}(y)+\int_{\mathcal{X}}k(y,x)\,g_{\lambda}(x)\,P(dx)=\int_{\mathcal{X}}k(y,x)\,g_{0}(x)\,P(dx),
\]
or $(\lambda+K)g_{\lambda}=Kg_{0}$, the assertion.

\end{proof}

\subsection{\label{sec:SAA}The discrete problem and ridge regression}

We now switch from the continuous problem~\eqref{eq:Genuine} to
learning from data. This alternative viewpoint highlights and justifies
the genuine estimator~\eqref{eq:Estimator} from an additional perspective.

Substituting the average for the expectation in~\eqref{eq:Genuine}
we consider the slightly more general objective 
\begin{equation}
\frac{1}{n}\sum_{i,j=1}^{n}\big(f_{i}-f(x_{i})\big)\Lambda_{ij}^{-1}\big(f_{j}-f(x_{j})\big)+\left\Vert f\right\Vert _{k}^{2},\label{eq:20}
\end{equation}
where $\Lambda$ is a symmetric and invertible regularization matrix.
We use lowercase letters $x_{i}\in\mathcal{X}$ and $f_{i}\in\mathbb{R}$
to emphasize that these quantities are deterministic.
\begin{prop}
The function $f\in\mathcal{H}_{k}$ minimizing~\eqref{eq:20} is
\begin{equation}
f(\cdot)=\frac{1}{n}\sum_{s=1}^{n}w_{i}\cdot k(\cdot,x_{i}),\label{eq:13-1}
\end{equation}
where the weights are
\begin{equation}
w=n\big(K^{\top}\Lambda^{-1}K+n\,K\big)^{-1}K^{\top}\Lambda^{-1}f.\label{eq:42-2}
\end{equation}
\end{prop}

\begin{proof}
Assuming that the optimal function is of the form~\eqref{eq:13-1},
the objective~\eqref{eq:20} is 
\[
\frac{1}{n}(f-\frac{1}{n}Kw)^{\top}\Lambda^{-1}\big(f-\frac{1}{n}Kw\big)+\frac{1}{n^{2}}w^{\top}Kw.
\]
Differentiating with respect to $w$ gives the first order conditions
\[
0=-\frac{1}{n^{2}}\left(K^{\top}\Lambda^{-1}\big(f-\frac{1}{n}Kw\big)\right)^{\top}-\frac{1}{n^{2}}\big(f-\frac{1}{n}Kw\big)^{\top}\Lambda^{-1}K+\frac{1}{n^{2}}(Kw)^{\top}+\frac{1}{n^{2}}w^{\top}K,
\]
i.e., 
\[
\frac{1}{n^{2}}\left(\frac{1}{n}K^{\top}\left(\Lambda^{-1}+\Lambda^{-\top}\right)K+K+K^{\top}\right)w=\frac{1}{n^{2}}K^{\top}\left(\Lambda^{-1}+\Lambda^{-\top}\right)f.
\]
The assertion follows, as $\Lambda^{-1}$ and $K$ are both symmetric.

It remains to demonstrate that the optimal function is indeed of the
form~\eqref{eq:13-1}, i.e., the optimal function $f\in\mathcal{H}_{k}$
is located exactly on the supporting points $x_{1},\dots,x_{n}$.
This, however, follows from the representer theorem, which \citet{Schoelkopf2001}
prove in the most general form.
\end{proof}
\begin{cor}
\label{cor:Estimator}The function $f\in\mathcal{H}_{k}$ minimizing
the objective 
\begin{equation}
\frac{1}{n}\sum_{i=1}^{n}\big(f_{i}-f(x_{i})\big)^{2}+\lambda\left\Vert f\right\Vert _{k}^{2}\label{eq:21}
\end{equation}
is $f(\cdot)\coloneqq\frac{1}{n}\sum_{j=1}^{n}w_{j}\,k(\cdot,x_{j})$
with weights $w=\bigl(\lambda+\frac{1}{n}K\bigr)^{-1}f$.
\end{cor}

\begin{proof}
The assertion is immediate with $\Lambda=\lambda\cdot I$, the diagonal
matrix with entries~$\lambda$ on its diagonal.
\end{proof}

\section{\label{sec:StatisticalProperties}Elementary statistical properties}

As above, let $(X_{i},f_{i})$, $i=1,\dots,n$, be independent samples
from a joint measure~$\rho$. We note that $X_{i}\sim P$ and the
integral operator~$K$ in~\eqref{eq:K} can be restated as 
\[
Kw(x)=\E k(x,X_{i})\,w(X_{i})=\E\big(k(X_{i},X_{j})\,w(X_{j})\mid X_{i}=x\big);
\]
we shall make frequent use of the latter relation.
\begin{defn}
For $(X_{i},f_{i})$, $i=1,\dots,n$, independent samples from a joint
distribution~$\rho$ define the estimator
\begin{equation}
\hat{\vartheta}_{n}\coloneqq\min_{\hat{f}_{n}(\cdot)}\frac{1}{n}\sum_{i=1}^{n}\big(f_{i}-\hat{f}_{n}(X_{i})\big)^{2}+\lambda\left\Vert \hat{f}_{n}\right\Vert _{k}^{2}.\label{eq:Theta}
\end{equation}
\end{defn}

It is evident that $\hat{\vartheta}_{n}$ is an $\mathbb{R}$\nobreakdash-valued
random variable, dependent on the samples $(X_{i},f_{i})$. Further,
the optimizer 
\begin{equation}
\hat{f}_{n}(\cdot)\coloneqq\frac{1}{n}\sum_{i=1}^{n}k(\cdot,X_{i})\,\hat{w}_{i}\label{eq:SAA}
\end{equation}
of~\eqref{eq:Theta} (cf.\ Corollary~\ref{cor:Estimator}) is a
random function, as it is supported by the samples $X_{i}$, $i=1,\dots,n$,
and the weights 
\begin{equation}
\hat{w}=\Big(\lambda+\frac{1}{n}K\Big)^{-1}f\label{eq:wEstimator}
\end{equation}
depend on all~$(X_{i},f_{i})$, $i=1,\dots,n$. Relating to the term
sample average approximation (SAA) in stochastic optimization we shall
refer to the estimators~$\hat{\vartheta}_{n}$ and $\hat{f}_{n}(\cdot)$
as the \emph{SAA estimators}.
\begin{example}
A simple example is given by employing the trivial design measure
$P=\delta_{x_{0}}$, where $x_{0}\in\mathcal{X}$ is a fixed point
and $\delta_{x_{0}}(A)\coloneqq\begin{cases}
1 & \text{if }x_{0}\in A,\\
0 & \text{else}
\end{cases}$ is the Dirac\textendash measure. It is easily seen that the estimator~\eqref{eq:SAA}
is the function $\hat{f}_{n}(\cdot)=\frac{k(\cdot,x_{0})}{\lambda+k(x_{0},x_{0})}\cdot\frac{1}{n}\sum_{i=1}^{n}f_{i}$.
It is thus clear that the estimator~$f_{n}(\cdot)$ is biased and
all results necessarily depend on~$\lambda$. 
\end{example}

\begin{lem}
The estimator $\hat{\vartheta}_{n}$ and its optimizer $\hat{f}_{n}$
are bounded with probability $1$. More explicitly, for $\varepsilon>0$,
it holds that $0\le\hat{\vartheta}_{n}\le\|f\|_{2}^{2}+\varepsilon$
for $n$ large enough a.s.\ and the optimizer $\hat{f}_{\lambda,n}$
in~\eqref{eq:Theta} satisfies 
\begin{equation}
\|\hat{f}_{\lambda,n}\|_{k}\le\frac{\|f\|_{2}+\varepsilon}{\sqrt{\lambda}}\label{eq:Ball}
\end{equation}
 for~$n$ large enough almost surely.
\end{lem}

\begin{proof}
Choose $\hat{f}_{\lambda}(\cdot)=0$ in~\eqref{eq:Theta} to see
that $0\le\hat{\vartheta}_{n}\le\frac{1}{n}\sum_{i=1}^{n}f_{i}^{2}$.
By the strong law of large numbers there is $N(\omega,\varepsilon)$
so that $0\le\hat{\vartheta}_{n}\le\frac{1}{n}\sum_{i=1}^{n}f_{i}^{2}\le\|f\|_{2}^{2}+\varepsilon$
for every $n\ge N(\omega,\varepsilon)$. Further, $\lambda\|\hat{g}_{n}\|_{k}^{2}\le\|f\|_{2}^{2}+\varepsilon$
a.s.\ for every reasonable and feasible estimator $\tilde{g}_{n}$
in~\eqref{eq:Theta} and hence the assertion.
\end{proof}
The following consistency result is originally demonstrated in \citet[Lemma~4.1]{NorkinPflug}
in a different context. 
\begin{thm}[{Cf.\ \citet[Lemma~4.1]{NorkinPflug} and \citet[Proposition~5.6]{RuszczynskiShapiro2009}}]
\label{thm:Consistency}The estimator~$\hat{\vartheta}_{n}$ is
downwards biased and monotone in expectation for increasing sample
sizes; more precisely, it holds that 
\[
0\le\E\hat{\vartheta}_{n}\le\E\hat{\vartheta}_{n+1}\le\vartheta^{*},
\]
where $\vartheta^{*}=\E\big(f-f_{\lambda}(X)\big)^{2}+\lambda\left\Vert f_{\lambda}\right\Vert _{k}^{2}$
with $f_{\lambda}(\cdot)$ given in~\eqref{eq:64} is the objective
of the contiuous  problem~\eqref{eq:Genuine} (see also~\eqref{eq:Theta*}).
\end{thm}

\begin{proof}
It holds that
\begin{align*}
\E\hat{\vartheta}_{n+1} & =\E\min_{\hat{f}_{n+1}(\cdot)}\frac{1}{n+1}\sum_{i=1}^{n+1}\big(f_{i}-\hat{f}_{n+1}(X_{i})\big)^{2}+\lambda\left\Vert \hat{f}_{n+1}\right\Vert _{k}^{2}\\
 & =\E\min_{\hat{f}_{n+1}(\cdot)}\frac{1}{n+1}\sum_{i=1}^{n+1}\frac{1}{n}\sum_{j\not=i}\big(f_{j}-\hat{f}_{n+1}(X_{j})\big)^{2}+\lambda\left\Vert \hat{f}_{n+1}\right\Vert _{k}^{2}\\
 & \ge\E\frac{1}{n+1}\sum_{i=1}^{n+1}\min_{\hat{f}_{i}(\cdot)}\frac{1}{n}\sum_{j\not=i}\big(f_{j}-\hat{f}_{i}(X_{j})\big)^{2}+\lambda\left\Vert \hat{f}_{i}\right\Vert _{k}^{2}\\
 & =\frac{1}{n+1}\sum_{i=1}^{n+1}\E\hat{\vartheta}_{n}=\E\hat{\vartheta}_{n}.
\end{align*}
Further, the optimal value of~\eqref{eq:Genuine} is given by~$f_{\lambda}$
(cf.~\eqref{eq:64} in Theorem~\ref{thm:Optimal}). 

Finally we have that 
\[
\min_{\hat{f}_{n}(\cdot)}\frac{1}{n}\sum_{i=1}^{n}\big(f_{i}-\hat{f}_{n}(X_{i})\big)^{2}+\lambda\left\Vert \hat{f}_{n}\right\Vert \le\frac{1}{n}\sum_{i=1}^{n}\big(f_{i}-\hat{f}_{n}(X_{i})\big)^{2}+\lambda\left\Vert \hat{f}_{n}\right\Vert .
\]
By taking expectations and the infimum afterwards we conclude that
$\E\hat{\vartheta}_{n}\le\vartheta^{*}$, the remaining inequality.
\end{proof}

\section{\label{sec:InNorm}Approximation in norm}

Recall that the optimal solution of the continuous problem~\eqref{eq:Genuine}
is the function  $f_{\lambda}(\cdot)\in\mathcal{H}_{k}$, while the
optimal solution of the discrete analogue~\eqref{eq:Theta} is the
random variable~\eqref{eq:SAA}. In what follows we shall establish
convergence of $\hat{f}_{n}(\cdot)$ towards $f_{\lambda}(\cdot)$
for increasing sample size~$n$.

To establish convergence in norm we relate the problems first to the
following auxiliary problem involving an auxiliary estimator~$\tilde{f}$.
Its residual constitutes an important relation between $f_{\lambda}$
and $\hat{f}_{n}$, but is unbiased itself. The auxiliary estimator~$\tilde{f}$
removes the bias and allows denoising the genuine problem. The Subsection~\ref{subsec:Relation}
below will reconnect the estimators~$\tilde{f}$ and $\hat{f}$. 

\subsection{Denoising and local bias adjustment}

The following estimator $\tilde{f}_{n}$ turns out to capture and
remove the noise in problem~\eqref{eq:Theta}.
\begin{defn}
Define the function
\begin{align*}
\tilde{f}_{n}(\cdot) & \coloneqq\frac{1}{n}\sum_{j=1}^{n}\tilde{w}_{j}\,k(\cdot,X_{j}),\quad\text{where }\tilde{w}_{j}\coloneqq\frac{f_{j}-f_{\lambda}(X_{j})}{\lambda},
\end{align*}
the residual function 
\begin{align}
\tilde{r}_{n}(\cdot) & \coloneqq f_{\lambda}(\cdot)-\tilde{f}_{n}(\cdot)=f_{\lambda}(\cdot)-\frac{1}{n}\sum_{j=1}^{n}k(\cdot,X_{j})\,\tilde{w}_{j}\label{eq:tilder}
\end{align}
and the vector of residuals with entries $\tilde{r}_{i}\coloneqq f_{i}-\lambda\,\tilde{w}_{i}-\frac{1}{n}\sum_{j=1}^{n}k(X_{i},X_{j})\,\tilde{w}_{j}$,
$i=1,\dots,n$.
\end{defn}

\begin{rem}
The weights~$\hat{w}_{i}$ and the function values~$f_{i}$ are
connected via the linear system of equations~\eqref{eq:Bandwidth}.
The visualization in Figure~\ref{fig:Correlation} indicates that
$f_{i}$ and the weights $\lambda\,\hat{w}_{i}$ are strongly correlated
with a gap approximately $f_{\lambda}(X_{i})$. The definition of
the auxiliary estimator~$\tilde{f}_{n}(\cdot)$ in the preceding
definition anticipates and explores this observation.

\begin{figure}
\subfloat[\label{fig:Gaussian}Gaussian field regression~$\hat{f}_{n}(\cdot)$
with range $\pm\sqrt{\hat{K}(\cdot)}$ given by~\eqref{eq:Bandwidth}]{\includegraphics[viewport=20bp 30bp 628bp 340bp,clip,width=0.51\textwidth]{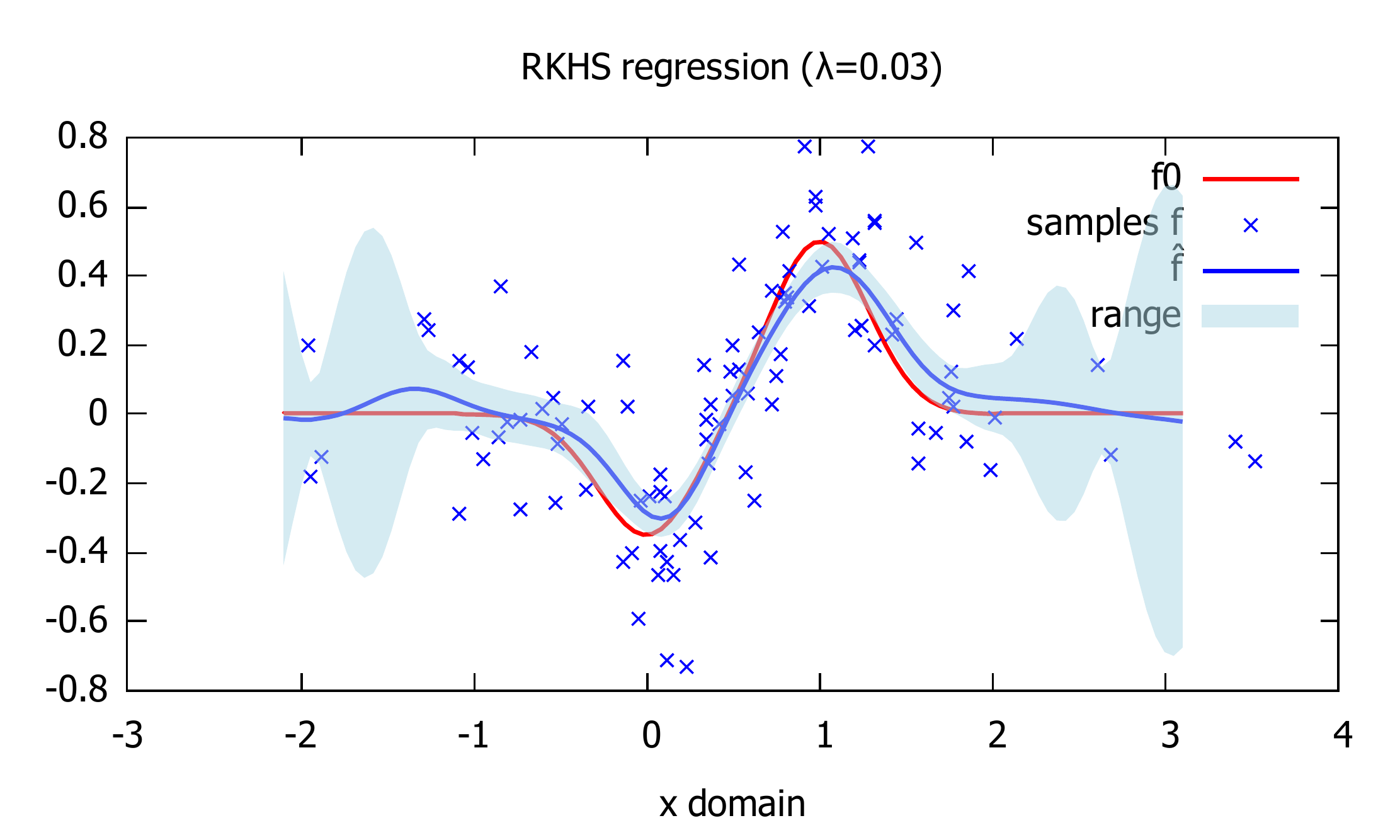}

}\subfloat[\label{fig:Correlation}The random weights $\lambda\,\hat{w}_{i}$
and the adapted noise $f_{i}-f_{\lambda}(X_{i})$ are almost comonotonic.]{\includegraphics[viewport=25bp 30bp 623bp 340bp,clip,width=0.49\textwidth]{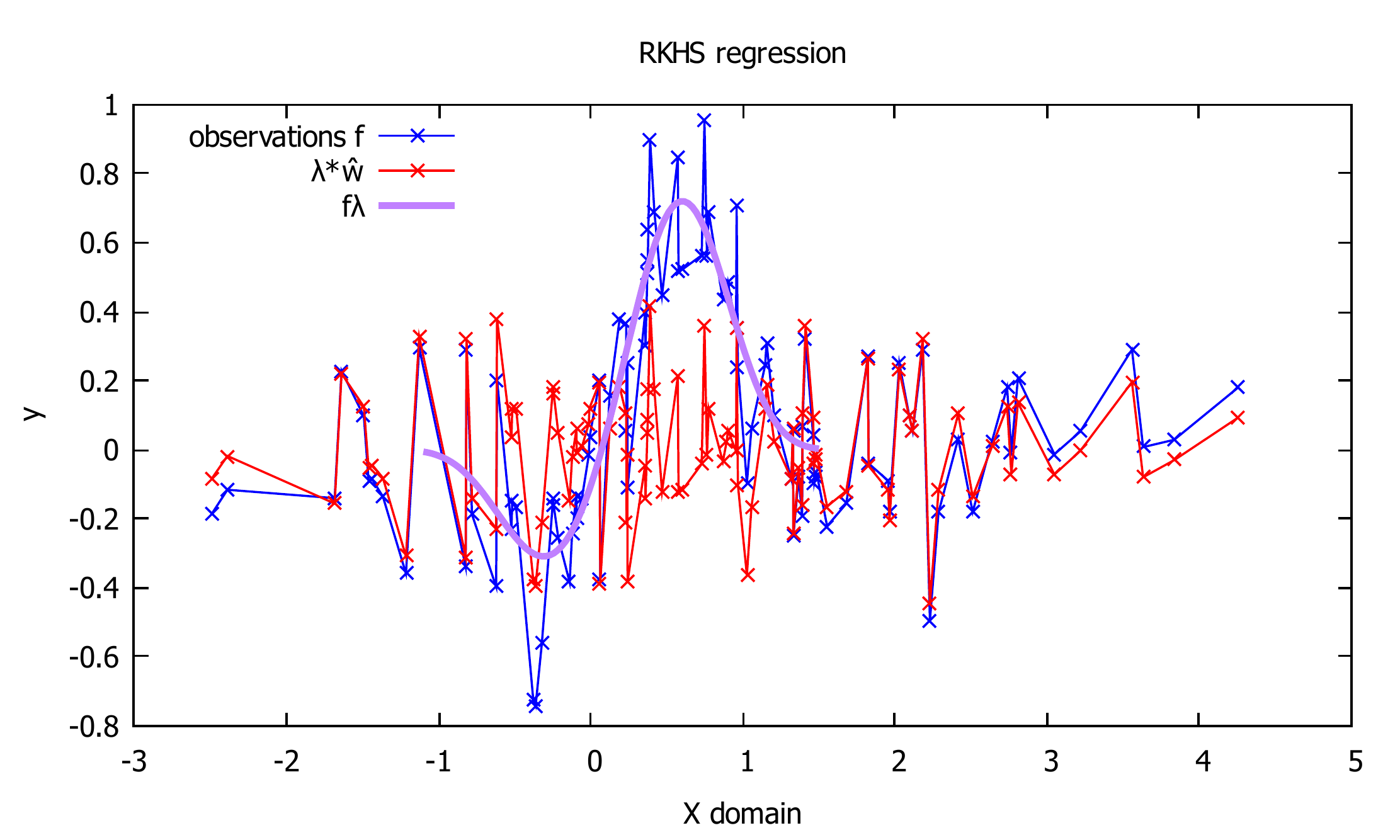}

}\caption{Approximation of functions for a sample of size $n=100$}
\end{figure}
\end{rem}

\begin{lem}
\label{lem:Residuals}The residuals are $\tilde{r}_{i}=\tilde{r}_{n}(X_{i})$.
\end{lem}

\begin{proof}
Indeed, 
\[
\tilde{r}_{i}=f_{i}-\lambda\,\frac{f_{i}-f_{\lambda}(X_{i})}{\lambda}-\frac{1}{n}\sum_{j=1}^{n}k(X_{i},X_{j})\,\tilde{w}_{j}=f_{\lambda}(X_{i})-\frac{1}{n}\sum_{j=1}^{n}k(X_{i},X_{j})\,\tilde{w}_{j}=\tilde{r}_{n}(X_{i}),
\]
the assertion.
\end{proof}
We shall establish the relation between $\tilde{f}_{n}(\cdot)$ and
$f_{\lambda}(\cdot)$ first. To this end recall that $\tilde{f}_{n}(\cdot)$
is random, while~$f_{\lambda}(\cdot)$ is deterministic. The function~$\tilde{f}_{n}(\cdot)$
recovers the function $f_{\lambda}(\cdot)$ \emph{on average} and
enjoys the following statistical properties.
\begin{prop}[$\tilde{f}_{n}$ is $f_{\lambda}$ on average]
\label{prop:Residual}It holds that 
\[
f_{\lambda}(x)=\E\tilde{f}_{n}(x),\qquad x\in\mathcal{X}.
\]
Equivalently, the residual $\tilde{r}(\cdot)$ is locally unbiased,
i.e., 
\[
\E\tilde{r}_{n}(x)=0
\]
 for every $x\in\mathcal{X}$.
\end{prop}

\begin{proof}
Observe first with~\eqref{eq:Ef} and~\eqref{eq:36} that 
\begin{equation}
\E\left(\left.\tilde{w}_{j}\right|X_{j}=x\right)=\E\left(\left.\frac{f_{j}-f_{\lambda}(X_{j})}{\lambda}\right|X_{j}=x\right)=\frac{f_{0}(x)-f_{\lambda}(x)}{\lambda}=w_{\lambda}(x)\label{eq:Weights}
\end{equation}
as both, $f_{j}$ and $f_{\lambda}(X_{j})$ are in $L^{2}$. By the
tower property of the expectation, by taking out what is known and~\eqref{eq:34},
\begin{equation}
\E k(x,X_{j})\,\tilde{w}_{j}=\E\left(k(x,X_{j})\E(\tilde{w}_{j}|\,X_{j})\right)=\E k(x,X_{j})\,w_{\lambda}(X_{j})=f_{\lambda}(x)\label{eq:46}
\end{equation}
for every $j=1,\dots,n$. With this, the assertion is immediate and
the expected value of the residual follows together with its definition
in~\eqref{eq:tilder}.
\end{proof}
The preceding relation reveals the expectation of $\tilde{f}_{n}$
locally. The next proposition demonstrates local convergence for increasing
sample size~$n$.
\begin{prop}[Local approximation quality]
For every $x\in\mathcal{X}$ there is a constant $C(x)>0$ so that
\[
\var\tilde{f}_{n}(x)=\frac{C(x)}{n}.
\]
\end{prop}

\begin{proof}
Employing Proposition~\ref{prop:Residual} we have that 
\begin{align*}
\var\tilde{f}_{n}(x) & =\E\Big(f_{\lambda}(x)-\frac{1}{n}\sum_{j=1}^{n}\tilde{w_{j}}\,k(x,X_{j})\Big)^{2}\\
 & =f_{\lambda}(x)^{2}-2f_{\lambda}(x)^{2}+\frac{1}{n^{2}}\sum_{i,j=1}^{n}\E\tilde{w}_{i}\,\tilde{w}_{j}\,k(x,X_{i})\,k(x,X_{j})\\
 & =f_{\lambda}(x)^{2}-2f_{\lambda}(x)^{2}+\frac{n^{2}-n}{n^{2}}f_{\lambda}(x)^{2}+\frac{n}{n^{2}}\sum_{i=1}^{n}\E\Big(\frac{f_{i}-f_{\lambda}(X_{i})}{\lambda}\Big)^{2}k(x,X_{i})^{2}
\end{align*}
as $X_{i}$ and $X_{j}$ are independent for $i\not=j$. It follows
that $\var\tilde{f}_{n}(x)=-\frac{1}{n}f_{\lambda}(x)^{2}+\frac{1}{n\,\lambda^{2}}C^{\prime}(x)$,
where $C^{\prime}(x)=\E\big(f-f_{\lambda}(X)\big)^{2}\,k(x,X)^{2}$
is finite.
\end{proof}

\subsection{Uniform approximation properties of the auxiliary estimator $\tilde{f}_{n}$}

The following theorem reveals the \emph{precise} approximation quality
of the estimator with weights~$\tilde{w}$.
\begin{thm}[Approximation in norm]
\label{thm:Norm1}It holds that 
\[
\E\|f_{\lambda}-\tilde{f}_{n}\|_{k}^{2}=\frac{C_{1}}{\lambda^{2}\,n},
\]
where $C_{1}>0$ is a constant independent on~$\lambda$ and~$n$.
More explicitly, 
\[
\E\left\Vert \tilde{r}_{n}\right\Vert _{k}^{2}=\frac{1}{\lambda^{2}\,n}\int_{\mathcal{X}}\Big(\var(f|\,x)+\big(f_{0}(x)-f_{\lambda}(x)\big)^{2}\Big)k(x,x)P(dx)-\frac{1}{n}\left\Vert f_{\lambda}\right\Vert _{k}^{2},
\]
where $\var\big(f|\,x\big)=\E\Big(\big(f-f_{0}(X)\big)^{2}\big|\,X=x\Big)$
is the variance of the random data at~$x$ (see~\eqref{eq:Ef}),
the local irreducible error.
\end{thm}

\begin{rem}
The conditional variance term~$\var(f|\,x)$ points to the fact that
convergence actually differs for homoscedastic and heteroscedastic
random observations $(X_{i},f_{i})$, $i=1,\dots,n$.
\end{rem}

\begin{proof}
With $f_{\lambda}=Kw_{\lambda}$ (cf.\ \eqref{eq:34}) we have that
\begin{align}
\MoveEqLeft[3]\E\biggl\Vert f_{\lambda}(\cdot)-\frac{1}{n}\sum_{j=1}^{n}k(\cdot,X_{j})\,\tilde{w}_{j}\biggr\Vert_{k}^{2}\nonumber \\
 & =\E\left\Vert f_{\lambda}\right\Vert _{k}^{2}-2\biggl\langle f_{\lambda}\biggl|\,\frac{1}{n}\sum_{j=1}^{n}k(\cdot,X_{j})\,\tilde{w}_{j}\biggr\rangle_{k}+\biggl\Vert\frac{1}{n}\sum_{j=1}^{n}k(\cdot,X_{j})\,\tilde{w}_{j}\biggr\Vert_{k}^{2}\nonumber \\
 & =\left\Vert f_{\lambda}\right\Vert _{k}^{2}\nonumber \\
 & \qquad-2\E\frac{1}{n}\sum_{j=1}^{n}\int_{\mathcal{X}}w_{\lambda}(y)\,k(y,X_{j})\,\tilde{w}_{j}\,P(dy)\label{eq:42-1}\\
 & \qquad+\frac{1}{n^{2}}\sum_{i,j=1}^{n}\E\tilde{w}_{i}\,k(X_{i},X_{j})\,\tilde{w}_{j}.\label{eq:43}
\end{align}
With~\eqref{eq:46} and~\eqref{eq:kNorm}, the term~\eqref{eq:42-1}
is 
\[
\E\frac{2}{n}\sum_{i=1}^{n}\int_{\mathcal{X}}w_{\lambda}(y)\,k(y,X_{j})\,\tilde{w}_{j}P(dy)=2\int_{\mathcal{X}}w_{\lambda}(y)\,f_{\lambda}(y)\,P(dy)=2\left\Vert f_{\lambda}\right\Vert _{k}^{2}.
\]
For the remaining term~\eqref{eq:43} involving all combinations
and by separating all combinations with $j=i$ from those with $j\not=i$
we find 
\begin{align*}
\MoveEqLeft[3]\frac{1}{n^{2}}\sum_{i,j=1}^{n}\E\tilde{w}_{i}\,k(X_{i},X_{j})\,\tilde{w}_{j}=\\
 & =\frac{1}{n^{2}}\sum_{i=1}^{n}\E\E\Big(\Big(\frac{f_{i}-f_{\lambda}(X_{i})}{\lambda}\bigg)^{2}k(X_{i},X_{i})\biggr|X_{i}\bigg)+\frac{1}{n}\sum_{i=1}^{n}\frac{1}{n}\E\E\bigg(\tilde{w}_{i}\sum_{j\not=i}k(X_{i},X_{j})\,\tilde{w}_{j}\biggr|X_{i}\bigg)\\
 & =\frac{1}{n^{2}\lambda^{2}}\sum_{i=1}^{n}\E\E\Big(\big(f_{i}-f_{0}(X_{i})\big)^{2}+\big(f_{0}(X_{i})-f_{\lambda}(X_{i})\big)^{2}\biggr|X_{i}\bigg)\,k(X_{i},X_{i})\\
 & \qquad+\frac{1}{n}\sum_{i=1}^{n}\frac{1}{n}\E\E\bigg(\left(\frac{f_{i}-f_{0}(X_{i})}{\lambda}+\frac{f_{0}(X_{i})-f_{\lambda}(X_{i})}{\lambda}\right)\sum_{j\not=i}k(X_{i},X_{j})\,\tilde{w}_{j}\biggr|X_{i}\bigg),
\end{align*}
as $2\E\Big(\big(f_{i}-f_{0}(X_{i})\big)\big(f_{0}(X_{i})-f_{\lambda}(X_{i})\big)\biggr|X_{i}\bigg)=0$
by the orthogonality relation~\eqref{eq:Orthogonal}.

With the kernel trick (see Mercer's theorem in Remark~\ref{rem:Mercer}),
and independence of $(X_{i},f_{i})$ from $(X_{j},f_{j})$ for $j\not=i$
we have that
\begin{align*}
\E\E\bigg(\frac{f_{i}-f_{0}(X_{i})}{\lambda}k(X_{i},X_{j})\,\tilde{w}_{j}\biggr|X_{i}\bigg) & =\sum_{\ell=1}^{\infty}\E\bigg(\frac{f_{i}-f_{0}(X_{i})}{\lambda}\sigma_{\ell}^{2}\phi_{\ell}(X_{i})\phi_{\ell}(X_{j})\,\tilde{w}_{j}\bigg)\\
 & =\sum_{\ell=1}^{\infty}\sigma_{\ell}^{2}\E\bigg(\frac{f_{i}-f_{0}(X_{i})}{\lambda}\phi_{\ell}(X_{i})\bigg)\E\bigg(\phi_{\ell}(X_{j})\,\tilde{w}_{j}\bigg).
\end{align*}
Again, by the orthogonality relation~\eqref{rem:Mercer} we concluded
that $\E\Big(\frac{f_{i}-f_{0}(X_{i})}{\lambda}\phi_{\ell}(X_{i})\Big)=0$.
Summing up we have 
\begin{align*}
\MoveEqLeft[3]\frac{1}{n^{2}}\sum_{i,j=1}^{n}\E\tilde{w}_{i}\,k(X_{i},X_{j})\,\tilde{w}_{j}=\\
 & =\frac{1}{n^{2}\lambda^{2}}\sum_{i=1}^{n}\E\E\Big(\big(f_{i}-f_{0}(X_{i})\big)^{2}+\big(f_{0}(X_{i})-f_{\lambda}(X_{i})\big)^{2}\biggr|X_{i}\bigg)\,k(X_{i},X_{i})\\
 & \qquad+\frac{1}{n}\sum_{i=1}^{n}\frac{1}{n}\E\frac{f_{0}(X_{i})-f_{\lambda}(X_{i})}{\lambda}\sum_{j\not=i}\E\bigg(k(X_{i},X_{j})\,\tilde{w}_{j}\biggr|X_{i}\bigg)\\
 & =\frac{1}{n^{2}\lambda^{2}}\sum_{i=1}^{n}\E\Big(\var(f\mid X_{i})+\big(f_{0}(X_{i})-f_{\lambda}(X_{i})\big)^{2}\bigg)k(X_{i},X_{i})\\
 & \qquad+\frac{1}{n}\sum_{i=1}^{n}\frac{1}{n}\sum_{j\not=i}\E w_{\lambda}(X_{i})\,f_{\lambda}(X_{i}),
\end{align*}
where we have employed~\eqref{eq:46} again. As above, we have again
that 
\[
\E w_{\lambda}(X_{i})f_{\lambda}(X_{i})=\left\Vert f_{\lambda}\right\Vert _{k}^{2}.
\]
Collecting terms we find that 
\begin{align*}
\MoveEqLeft[3]\E\biggl\Vert f_{\lambda}(\cdot)-\frac{1}{n}\sum_{j=1}^{n}k(\cdot,X_{j})\,\tilde{w}_{j}\biggr\Vert_{k}^{2}=\left\Vert f_{\lambda}\right\Vert _{k}^{2}-2\left\Vert f_{\lambda}\right\Vert _{k}^{2}\\
 & +\frac{n-1}{n}\left\Vert f_{\lambda}\right\Vert _{k}^{2}+\frac{1}{\lambda^{2}n}\int_{\mathcal{X}}\Big(\var(f|\,x)+\big(f_{0}(x)-f_{\lambda}(x)\big)^{2}\Big)k(x,x)P(dx)
\end{align*}
and thus the assertion.
\end{proof}
\begin{rem}[Local correlation]
The coefficients $\tilde{w}_{i}$ depend explicitly on~$f_{i}$.
This explicit relation will actually allow us to dampen, even to remove
the noise from the estimators. Indeed, the noise~$f_{i}$ and the
coefficients~$\tilde{w}_{i}$ are utmost correlated, it holds that
\begin{equation}
\corr(f_{i},\tilde{w}_{i}\mid X_{i}=x)=1.\label{eq:Corr}
\end{equation}

To accept this strong correlation property recall the relation $\tilde{w}_{i}=\frac{1}{\lambda}f_{i}+\frac{1}{\lambda}f_{\lambda}(X_{i})$,
which exhibits\textemdash provided that $X_{i}=x$ is kept fixed\textemdash an
affine relation between $\tilde{w}_{i}$ and $f_{i}$ and thus~\eqref{eq:Corr}.

\end{rem}

\subsection{\label{subsec:Relation}The relation of the SAA estimator $\hat{f}_{n}$
and $\tilde{f}_{n}$}

The unbiased estimator $\tilde{f}_{n}$ and the estimator of interest
$\hat{f}_{n}$ are connected explicitly in the following way.
\begin{lem}
It holds that\footnote{$\big(\lambda+\frac{1}{n}K\big){}_{j}^{-1}$ is the $j$\nobreakdash-row
(or column, as $K$ is symmetric) of the matrix $\big(\lambda+\frac{1}{n}K\big){}^{-1}$.} 
\[
\hat{f}_{n}(\cdot)-\tilde{f}_{n}(\cdot)=\frac{1}{n}\sum_{j=1}^{n}\tilde{r}_{n}^{\top}\Big(\lambda+\frac{1}{n}K\Big)_{j}^{-1}\,k(\cdot,X_{j})
\]
and
\begin{equation}
\|\hat{f}_{n}-\tilde{f}_{n}\|_{k}^{2}=\frac{1}{n}\tilde{r}^{\top}\Big(\lambda+\frac{1}{n}K\Big)^{-1}\,\frac{1}{n}K\,\Big(\lambda+\frac{1}{n}K\Big)^{-1}\tilde{r};\label{eq:13}
\end{equation}
here, $K$ is the random Gramian matrix with entries $K_{ij}=k(X_{i},X_{j})$.
\end{lem}

\begin{proof}
By the definition of $\tilde{r}_{i}$ and~\eqref{eq:wEstimator},
\begin{align*}
\tilde{r}_{i} & =f_{i}-\lambda\,\tilde{w}_{i}-\frac{1}{n}\sum_{j=1}^{n}k(X_{i},X_{j})\,\tilde{w}_{j}\\
 & =\lambda\,\hat{w}_{i}+\frac{1}{n}\sum_{j=1}^{n}k(X_{i},X_{j})\hat{w}_{j}-\lambda\,\tilde{w}_{i}-\frac{1}{n}\sum_{j=1}^{n}k(X_{i},X_{j})\,\tilde{w}_{j}\\
 & =\Big(\lambda+\frac{1}{n}K\Big)_{i}(\hat{w}-\tilde{w})
\end{align*}
and thus 
\[
\hat{w}-\tilde{w}=\Big(\lambda+\frac{1}{n}K\Big)^{-1}\tilde{r}.
\]
Now recall that $\hat{f}_{n}(\cdot)-\tilde{f}_{n}(\cdot)=\frac{1}{n}\sum_{j=1}^{n}(\hat{w}_{j}-\tilde{w}_{j})\,k(\cdot,X_{j})$
and the definition of the inner product~$\left\langle \cdot\mid\cdot\right\rangle _{k}$
to accept the remaining assertion.
\end{proof}
In what follows we provide the relation between the estimator of interest~$\hat{f}$
and the auxiliary estimator~$\tilde{f}$. The following Lemma is
essential, it allows to get rid of the random matrix $K$ and its
inverse in~\eqref{eq:13}.
\begin{lem}
\label{lem:2}For any nonnegative definite matrix $K$ (i.e., $K\ge0$)
it holds that 
\begin{equation}
\big(\lambda+K\big)^{-1}K\big(\lambda+K\big)^{-1}\le\frac{1}{4\lambda}\label{eq:19}
\end{equation}
 in Loewner order.
\end{lem}

\begin{proof}
It holds that $0\le(\lambda-K)^{2}=\lambda^{2}-2\lambda K+K^{2}$
and thus $4\lambda K\le\lambda^{2}+2\lambda K+K^{2}=(\lambda+K)^{2}$.
The assertion follows after multiplying with the corresponding inverse
from left and right.
\end{proof}
\begin{prop}
\label{prop:Norm2}It holds that
\begin{equation}
\E\|\hat{f}_{n}-\tilde{f}_{n}\|_{k}^{2}\le\frac{C_{2}}{4\lambda^{3}n}\label{eq:57}
\end{equation}
for a constant~$C_{2}>0$ independent of~$\lambda$ and~$n$.
\end{prop}

\begin{proof}
From~\eqref{eq:13} and Lemma~\ref{lem:2}, applied to the matrix
$\frac{1}{n}K$, we conclude that 
\begin{equation}
\|\hat{f}_{n}-\tilde{f}_{n}\|_{k}^{2}=\frac{1}{n}\tilde{r}^{\top}\Big(\lambda+\frac{1}{n}K\Big)^{-1}\frac{1}{n}K\Big(\lambda+\frac{1}{n}K\Big)^{-1}\tilde{r}\le\frac{1}{4\lambda}\cdot\frac{1}{n}\tilde{r}^{\top}\tilde{r}.\label{eq:Res1}
\end{equation}
With Lemma~\ref{lem:Residuals} it follows that 
\[
\E\|\hat{f}_{n}-\tilde{f}_{n}\|_{k}^{2}\le\frac{1}{4\lambda}\cdot\frac{1}{n}\sum_{i=1}^{n}\E\tilde{r}_{n}(X_{i})^{2}=\frac{1}{4\lambda}\cdot\E\tilde{r}_{n}(X_{i})^{2}
\]
for any $i=1,\dots,n$. Employing the definition of $\tilde{r}_{n}(\cdot)$
(cf.\ \eqref{eq:tilder}) the right hand side expression expands
as 
\begin{align}
\E\tilde{r}_{n}(X_{i})^{2} & =\E\left(f_{\lambda}(X_{i})-\frac{1}{n}\sum_{j=1}^{n}k(X_{i},X_{j})\frac{f_{j}-f_{\lambda}(X_{j})}{\lambda}\right)^{2}\nonumber \\
 & =\E f_{\lambda}(X_{i})^{2}\nonumber \\
 & \qquad-\frac{2}{n}\sum_{j=1}^{n}\E f_{\lambda}(X_{i})\cdot k(X_{i},X_{j})\frac{f_{j}-f_{\lambda}(X_{j})}{\lambda}\label{eq:40}\\
 & \qquad+\frac{1}{n^{2}}\sum_{j,\ell=1}^{n}\E k(X_{i},X_{j})\frac{f_{j}-f_{\lambda}(X_{j})}{\lambda}k(X_{i},X_{\ell})\frac{f_{\ell}-f_{\lambda}(X_{\ell})}{\lambda}.\label{eq:41}
\end{align}
We now treat~\eqref{eq:40} and~\eqref{eq:41} separately. As for
the first term we have for $j\not=i$ that 
\begin{align*}
\MoveEqLeft[3]\E f_{\lambda}(X_{i})\cdot k(X_{i},X_{j})\frac{f_{j}-f_{\lambda}(X_{j})}{\lambda}\\
 & =\E f_{\lambda}(X_{i})\E\bigg(\left.k(X_{i},X_{j})\frac{f_{j}-f_{\lambda}(X_{j})}{\lambda}\right|X_{i},X_{j}\bigg)\\
 & =\E f_{\lambda}(X_{i})\E\bigg(\left.k(X_{i},X_{j})\frac{f_{0}(X_{j})-f_{\lambda}(X_{j})}{\lambda}\right|X_{i}\bigg)\\
 & =\E f_{\lambda}(X_{i})\cdot\E\bigg(\left.k(X_{i},X_{j})\,w_{\lambda}(X_{j})\right|X_{i}\bigg)\\
 & =\E f_{\lambda}(X_{i})\cdot f_{\lambda}(X_{i})\\
 & =\E f_{\lambda}(X_{i})^{2},
\end{align*}
where we have used~\eqref{eq:36}. For $j=i$, the term~\eqref{eq:40}
is 
\[
\E f_{\lambda}(X_{i})k(X_{i},X_{i})\frac{f_{0}(X_{i})-f_{\lambda}(X_{i})}{\lambda}=\E f_{\lambda}(X_{i})k(X_{i},X_{i})w_{\lambda}(X_{i}).
\]

For the remaining term~\eqref{eq:41} and $i$, $j$ and $\ell$
all distinct we find that 
\begin{align*}
\MoveEqLeft[3]\E\E\bigg(\left.k(X_{i},X_{j})\frac{f_{j}-f_{\lambda}(X_{j})}{\lambda}k(X_{i},X_{\ell})\frac{f_{\ell}-f_{\lambda}(X_{\ell})}{\lambda}\right|X_{i}\bigg)\\
 & =\E\E\bigg(\left.k(X_{i},X_{j})\frac{f_{0}(X_{j})-f_{\lambda}(X_{j})}{\lambda}\right|X_{i}\bigg)\E\bigg(\left.k(X_{i},X_{\ell})\frac{f_{0}(X_{\ell})-f_{\lambda}(X_{\ell})}{\lambda}\right|X_{i}\bigg)\\
 & =\E\E\bigg(\left.k(X_{i},X_{j})w_{\lambda}(X_{j})\right|X_{i}\bigg)\E\bigg(\left.k(X_{i},X_{\ell})w_{\lambda}(X_{\ell})\right|X_{i}\bigg)\\
 & =\E f_{\lambda}(X_{i})^{2},
\end{align*}
as $X_{j}$ and $X_{\ell}$ are independent.

There are $2n-1$ out of $n^{2}$ combinations when not all $i$,
$j$ and $\ell$ are distinct, in which case the expression~\eqref{eq:41}
is finite with factor~$\frac{1}{\lambda^{2}}$ only if $j=\ell$.
Collecting now all terms and connecting with~\eqref{eq:Res1} reveals
the assertion~\eqref{eq:57} of the theorem.
\end{proof}
The elementary relation in Lemma~\ref{lem:2} is of crucial importance
in the preceding proof, as it allows to get rid of the random matrices
$\lambda+\frac{1}{n}K$ and, even more importantly, its inverse. We
discuss some situations, where the bound can be improved.
\begin{rem}
Suppose the kernel function is uniformly bounded from below, 
\begin{equation}
k(\cdot,:)\ge k_{\mathit{min}}>0\label{eq:kMin}
\end{equation}
$P^{2}$\nobreakdash-almost everywhere on $\mathcal{X}\times\mathcal{X}$.
Then the assertion of Proposition~\ref{prop:Norm2} is
\[
\E\|\hat{f}_{n}-\tilde{f}_{0}\|_{k}\le\frac{C_{2}}{k_{\mathit{min}}\,\lambda^{2}\,n}.
\]

Indeed, observe first that $k_{\mathit{min}}\,K\le K^{2}\le(\lambda+K)(\lambda+K)$
so that 
\begin{equation}
\big(\lambda+K\big)^{-1}K\big(\lambda+K\big)^{-1}\le\frac{1}{k_{\mathit{min}}}.\label{eq:18}
\end{equation}
The same proof as Proposition~\ref{prop:Norm2} applies (with~\eqref{eq:19}
replaced by~\eqref{eq:18}) and we conclude with 
\begin{equation}
\E\|\hat{f}_{n}-\tilde{f}_{n}\|_{k}^{2}\le\frac{1}{k_{\mathit{min}}\,\lambda^{2}n}.\label{eq:56}
\end{equation}
\end{rem}

Note that the assumption particularly implies that the support $\operatorname{supp}P$
is compact in $\mathcal{X}$ (see Footnote~\ref{fn:Support} on page~\pageref{fn:Support}).
But kernel functions~$k$, which are not compactly supported, enjoy
the property~\eqref{eq:kMin} on compact subsets of~$\mathcal{X}$
so that this assumption is not unusual in applications.
\begin{rem}
The inequality~\eqref{eq:19} is crucial in the analysis above as
it allows to get rid of the inverse of a matrix with random coefficients.
We assume that a better estimate at this point will likely improve
the quality of the approximation as in~\eqref{eq:56}; perhaps the
estimates on eigenvalues presented by \citet{ShaweTaylor2005} can
be of help to improve the inequality.
\end{rem}

\section{\label{sec:Convergence}Convergence in norm and consistency }

We can now connect the auxiliary and partial results of the preceding
sections to present our main results. They identify the limit in the
initial problem~\eqref{eq:Master} and describe convergence of the
estimator $\hat{f}_{n}$ towards $f_{\lambda}$ and towards $f_{0}$,
as well as consistency of the estimators.

\subsection{Convergence in norm}

The estimator $\hat{f}_{n}$ converges to $f_{\lambda}$ in expected
norm, as the sample size increases.
\begin{thm}
\label{thm:6}For the estimator $\hat{f}_{n}(\cdot)$ it holds that
$\E\|\hat{f}_{n}-f_{\lambda}\|_{k}^{2}\le\frac{2C_{1}}{\lambda^{2}\,n}+\frac{C_{2}}{\lambda^{3}\,n}$,
where $C_{1}$ and $C_{2}>0$ are constants independent of $\lambda$
and $n$.
\end{thm}

\begin{proof}
By the triangle inequality, Theorem~\ref{thm:Norm1} and Proposition~\ref{prop:Norm2}
we find that 
\begin{align*}
\E\|\hat{f}_{n}-f_{\lambda}\|_{k}^{2} & \le\E\big(\|f_{\lambda}-\tilde{f}_{n}\|_{k}+\|\tilde{f}_{n}-\hat{f}_{n}\|_{k}\big)^{2}\\
 & \le2\E\|f_{\lambda}-\tilde{f}_{n}\|_{k}^{2}+2\E\|\tilde{f}_{n}-\hat{f}_{n}\|_{k}^{2}\\
 & \le\frac{2C_{1}}{\lambda^{2}\,n}+\frac{C_{2}}{2\lambda^{3}\,n},
\end{align*}
the assertion.
\end{proof}
\begin{cor}[Convergence in $L^{2}$]
It holds that $\E\|\hat{f}_{n}-f_{\lambda}\|_{2}^{2}\le\frac{2C_{1}}{\lambda^{2}\,n}+\frac{C_{2}}{\lambda^{3}\,n}$,
where $C_{1}$ and $C_{2}>0$ are constants independent of $\lambda$
and $n$.
\end{cor}

\begin{proof}
The assertion is immediate with Proposition~\ref{prop:Norm}.
\end{proof}
\begin{thm}
\label{thm:63}Suppose that $f_{0}=Kw_{0}$ with $w_{0}\in\mathcal{H}_{k}$
so that $f_{0}$ is in the Hilbert space $\mathcal{H}_{k}$ as well.
Then there are constants $C_{0}$, $C_{1}$ and $C_{2}$, all independent
of $n$ and $\lambda$, so that 
\begin{equation}
\E\|f_{0}-\hat{f}_{n}\|_{k}^{2}\le2C_{0}\lambda^{2}+\frac{4C_{1}}{\lambda^{2}\,n}+\frac{C_{2}}{\lambda^{3}\,n}.\label{eq:62-1}
\end{equation}
\end{thm}

\begin{proof}
By involving Proposition~\ref{prop:f0} in combination with Theorem~\ref{thm:6}
we have that 
\begin{align*}
\E\|f_{0}-\hat{f}_{n}\|_{k}^{2} & \le\E\big(\|f_{0}-f_{\lambda}\|_{k}+\|\hat{f}_{n}-f_{\lambda}\|_{k}\big)^{2}\\
 & \le2\|f_{0}-f_{\lambda}\|_{k}^{2}+2\E\|\hat{f}_{n}-f_{\lambda}\|_{k}^{2}\\
 & \le2\lambda^{2}\|w_{0}\|_{k}^{2}+\frac{4C_{1}}{\lambda^{2}\,n}+\frac{C_{2}}{\lambda^{3}\,n}
\end{align*}
and hence the assertion with $C_{0}\coloneqq\|w_{0}\|_{k}^{2}$.
\end{proof}
As above, we have the following corollary.
\begin{cor}[Convergence in $L^{2}$]
For $f_{0}=Kw_{0}\in\mathcal{H}_{k}$ there are constants $C_{0}$,
$C_{1}$ and $C_{2}$, all independent of $n$ and $\lambda$, so
that that
\[
\E\|f_{0}-\hat{f}_{n}\|_{2}^{2}\le C_{0}\,\lambda^{2}+\frac{4C_{1}}{\lambda^{2}\,n}+\frac{C_{2}}{\lambda^{3}\,n}.
\]
\end{cor}

\begin{proof}
Just recall that the norm in $L^{2}$ is $\|f\|_{2}^{2}=\E|f|^{2}$.
\end{proof}

\subsection{Asymptotically optimal convergence rates and uniform approximation}

The results in the preceding section exhibit the typical bias variance
problem: the parameter~$\lambda$ in~\eqref{eq:62-1}, for example,
should be small to increase the approximation quality of $f_{\lambda}$
for $f_{0}$; on the other side, $\lambda$ should be large to improve
the approximation of $f_{\lambda}$ and the estimator $\hat{f}_{n}$.
The following statements reveal the best approximation rates asymptotically.
\begin{thm}
For $f_{0}$ in the range of $K$ and $\lambda_{n}=C\cdot n^{-\nicefrac{1}{5}}$
it holds that 
\[
\E\|f_{0}-\hat{f}_{n}\|_{k}^{2}\le\frac{C}{n^{\nicefrac{2}{5}}}.
\]
\end{thm}

\begin{proof}
The assertion derives from~\eqref{eq:62-1}.
\end{proof}
The following corollary is again immediate with Proposition~\ref{prop:Norm}.
\begin{cor}[Convergence in $L^{2}$]
For $f_{0}=Kw_{0}$ it holds that 
\[
\E\|f_{0}-\hat{f}_{n}\|_{2}^{2}\le\mathcal{O}\big(n^{-\nicefrac{2}{5}}\big),
\]
provided that $\lambda_{n}=\mathcal{O}\big(n^{-\nicefrac{1}{5}}\big)$.
\end{cor}

\begin{rem}
Assuming~\eqref{eq:kMin} we found the slower rate~\eqref{eq:56}.
With that, the leading term is $\frac{C}{\lambda^{2}n}$ instead of
$\frac{C}{\lambda^{2}n}$ in~\eqref{eq:62-1} and the optimal rate
is 
\[
\E\|f_{0}-\hat{f}_{n}\|_{k}^{2}\le\frac{C}{n^{\nicefrac{1}{2}}}
\]
and by Jensen's inequality and~\eqref{eq:Point} thus 
\[
C_{k}^{-1}\|f_{0}-\hat{f}_{n}\|_{\infty}\le\E\|f_{0}-\hat{f}_{n}\|_{k}\le\frac{C}{n^{\nicefrac{1}{4}}}
\]
for $\lambda=\mathcal{O}\big(n^{-\nicefrac{1}{2}}\big)$, where $C_{k}\coloneqq\sup_{x\in\operatorname{supp}P}\sqrt{k(x,x)}$.
\end{rem}

\subsection{\label{sec:Consistency}Weak consistency}

We have seen in Theorem~\ref{thm:Consistency} that the estimator
$\hat{\vartheta}_{n}$ of the objective is downwards biased. However,
weak consistency of the estimator $\hat{\vartheta}_{n}$ is immediate
as the optimizers converge.
\begin{thm}
\label{thm:71}Given the conditions of Theorem~\ref{thm:63} it holds
that~$\hat{f}_{n}$ converges to $f_{0}$ in probability. Further,
for every $x\in\mathcal{X}$, $\hat{f}_{n}(x)\to f_{0}(x)$, as $n\to\infty$,
in probability.
\end{thm}

\begin{proof}
Indeed, by Markov's inequality, 
\[
P\big(\|f_{0}-\hat{f}_{n}\|_{k}\ge\varepsilon\big)\le\frac{1}{\varepsilon^{2}}\E\|f_{0}-\hat{f}_{n}\|_{k}^{2}\to0,
\]
as $n\to\infty$ and thus the assertion is immediate.
\end{proof}
\begin{thm}
The estimators $\hat{\vartheta}_{n}$ are $L^{2}$\nobreakdash-consistent.
\end{thm}

\begin{proof}
The assertion is immediate by Theorem~\ref{prop:Norm} and the fact
that $\hat{f}_{n}$ is optimal for~$\hat{\vartheta}_{n}$ in~\eqref{eq:Theta}.
\end{proof}

\section{\label{sec:Summary}Discussion and summary}

\todo{Stochastische Optimierung, infty norm}

This paper addresses the regression problem to learn or reconstruct
a function, the conditional expectation function, from data observed
with noise. The method investigates an unbiased functional estimator,
which reconstructs the desired function under general preconditions.
This estimator is closely related to a popular estimator employed
in machine learning, for which we develop a tight relation. We provide
results for convergence in the norm of the genuine space, the norm
associated with the reproducing kernel Hilbert space.

The norm of the reproducing kernel Hilbert space is stronger than
uniform convergence. For this reason, the results allow to estimate
functions and establish their uniform convergence. With that, the
results are just appropriate for applications in stochastic optimization,
a subject with many intersections with neural networks and deep learning. 

\medskip{}

The convergence rates presented here are in line with other results
in nonparametric statistics. However, we believe to have evidence
from numerical computations that convergence rates can be improved
and this is subject to forthcoming research. A further topic, which
this paper does not touch, is the selection of the bandwidth. As well
it would be interesting to find and characterize the limiting distribution.

Some of the results can be compared with the Nadaraya\textendash Watson
estimator (see \citet{Tsybakov} on kernel density estimation), which
builds on kernels as well to estimator the conditional expectation.
This method from nonparametric statistics has similar convergence
properties and requires an oracle on the density function to find
optimal convergence rates. 

Finally we want to mention that we have an implementation available
at github,
\begin{center}
\href{https://github.com/aloispichler/reproducing-kernel-Hilbert-space}{https://github.com/aloispichler/reproducing-kernel-Hilbert-space},
\par\end{center}

\noindent which allows assessing the theoretical results of the paper
numerically.

\section{Acknowledgment}

We wish to thank \href{https://sites.gatech.edu/alexander-shapiro/}{Prof.\ Alexander Shapiro},
Georgia Tech, and \href{https://www.tu-chemnitz.de/mathematik/ang_analysis/prof.php}{Prof.\ Tino Ullrich},
TU Chemnitz, for discussion on a draft version of the manuscript.

\bibliographystyle{abbrvnat}
\bibliography{../../Literatur/LiteraturAlois}

\end{document}